\documentclass[11pt, twoside, leqno]{article}

\usepackage{latexsym}
\usepackage{amsmath}
\usepackage{amsfonts}
\usepackage{amssymb}
\usepackage{amsthm}
\usepackage{amscd}
\usepackage{color}
\usepackage[all]{xy}
\usepackage{diagrams}

\swapnumbers \theoremstyle{plain}
\newtheorem{thm}{Theorem}[section]

\newtheorem{pro}[thm]{Proposition}

\newtheorem{cor}[thm]{Corollary}

\theoremstyle{definition}
\newtheorem{Def}[thm]{Definition}

\newtheorem{exa}[thm]{Example}
\newtheorem{exas}[thm]{Examples}

\numberwithin{equation}{section}

\newcommand{\lra}{\longrightarrow}               %%arrows

\newcommand{\ra}{\rightarrow}

\newcommand{\lmt}{\longmapsto}

\newcommand{\sst}{\subseteq}                  %%abrev. of symbols
\newcommand{\prt}{\partial}

\newcommand{\8}{{\mathcal C}^{\infty}}
\newcommand{\id}{{\rm id}}

\def\1q{\quad}
\def\2q{\quad\quad}
\def\3q{\quad\quad\quad}
\newcommand{\DT}{{\mathcal D}{\mathcal T}}
\newcommand{\morf}{(f, f_{\cA}, f_{\Om})}
\newcommand{\morg}{(g, g_{\cA}, g_{\Om})}
\newcommand{\dtr}[1]{(\cA_{#1}, \prt_{#1}, \Om_{#1})}

\newcommand{\A}{\mathbb{A}}             %%font Bbb. Com mathbb needs not

\newcommand{\C}{\mathbb{C}}

\newcommand{\N}{\mathbb{N}}
\newcommand{\R}{\mathbb{R}}

\newcommand{\slb}[1]{\slshape \bfseries #1}

\newcommand{\cA}{{\mathcal A}}                   %%calligraphic letters

\newcommand{\cC}{{\mathcal C}}

\newcommand{\cM}{{\mathcal M}}

\newcommand{\cS}{{\mathcal S}}
\newcommand{\cT}{{\mathcal T}}

\newcommand{\gra}{\alpha}

\newcommand{\grd}{\delta}
\newcommand{\gre}{\varepsilon}

\newcommand{\Om}{\Omega}
\newcommand{\om}{\omega}

\newcommand{\eeq}{\end{equation}}
\newcommand{\beq}{\begin{equation}}
\newcommand{\eal}{\end{align*}}
\newcommand{\bal}{\begin{align*}}
\newcommand{\eqn}{\end{eqnarray}}
\newcommand{\bqn}{\begin{eqnarray}}
\newcommand{\ega}{\end{gather*}}
\newcommand{\bga}{\begin{gather*}}

\begin{document}

\pagestyle{myheadings} \markboth{\centerline{\small{\sc
            M.~Fragoulopoulou and M.~H.~Papatriantafillou}}}
         {\centerline{\small{\sc Smooth manifolds vs differential triads}}}
\title{\bf Smooth manifolds vs differential triads}
\footnotetext{Keywords and phrases: Smooth manifold, sheaf, sheafification, differential triad, morphism of differential triads, differential, $Q-$algebra, (topological) spectrum.}

\footnotetext{Mathematics Subject Classification (2010): 18F15, 18F20, 54B40, 46J05, 53Z05}
\author{\bf M.~Fragoulopoulou and M.~H.~Papatriantafillou}

\date{}

\maketitle
%%%%%%%%%%%%%%%%%%%%%%%%%%%%%%%%%%%%%%%%%%%%%%%%%%%%%%%%%%%%%%%%%%%%%%%%%%%%

\begin{center}
\emph{Dedicated to Professor Serban Stratila, on the occasion of his seventieth birthday}
\end{center}
\medskip

\begin{abstract}
We consider differentiable maps in the setting of {\em Abstract Differential
Geometry} and we study the conditions that ensure the uniqueness of
differentials in this setting. In particular, we prove that smooth maps between smooth manifolds admit a unique differential, coinciding with the usual one. Thus smooth manifolds form a full subcategory of the category of
differential triads, a result with physical implications.
\end{abstract}

\maketitle

\section{Introduction}

Circa 1990 A. Mallios used sheaf--theoretic methods to extend the {\em
mechanism} of the classical differential geometry (CDG) of smooth manifolds to
spaces, which do not admit the usual smooth structure (: smooth atlas); see \cite{MAL;Abstract, MAL;Note}. In this new setting of abstract
differential geometry (ADG) a large number of notions and results of CDG have
already been extended (\cite{MAL;VS, VASS;GPS}), becoming at the same time applicable to spaces with singularities and to quantum physics (see for instance, \cite{MAL;JTP, MAL-RAP, MAL-ROS2, MAL-ZAF}).

In ADG, the ordinary structure sheaf of smooth functions is replaced by a sheaf of abstract algebras $\cA$, admitting a differential $\prt$ (in the algebraic
sense), which takes values in an $\cA-$module $\Om$. A triplet $\dtr{}$ like
that is called a {\em differential triad}. Suitably defined morphisms organize
the differential triads into a category denoted by $\DT$ (\cite{MP;Mor}). Every smooth manifold defines a differential triad and every smooth map between
manifolds defines a morphism of the respective differential triads, so that the category $\cM an$ of smooth manifolds is embedded in $\DT$ (ibid.).

In the present paper we study the conditions assuring that a morphism in $\DT$
over a differentiable (in the abstract sense) map is uniquely determined, a
situation analogous to the classical ``uniqueness of differentials''.
Especially we prove that {\em a continuous map between manifolds, which is
differentiable in $\DT$, is also smooth in the usual sense, and its abstract
differential coincides with the ordinary one} (Theorem 4.5). This result makes differentials of maps between manifolds unique in both the abstract and the classical setting, while $\cM an$ becomes a {\em full subcategory} of $\DT$ (Theorem 4.6).

As a consequence, we have that phenomena described via CDG, have exactly the same interpretation in the more general setting of ADG.

\section{Preliminaries}

For the terminology applied the reader is mainly referred to \cite{MAL;Abstract}. However, for the reader's convenience, we recall the basic notions we use throughout the paper as those of differential triads and
of their morphisms; we also give a brief account of the way they form a category. Note that \emph{a smooth manifold will always be considered finite dimensional and 2nd countable}. Moreover, {\em all algebras considered are unital, commutati\-ve, associative and over the field $\C$ of complex numbers}; their units are denoted by $1$.

\begin{Def}[\cite{MAL;VS}]
Let $X$ be a topological space. A {\em differential triad} over $X$ is a
triplet $\delta = \dtr{}$, where $\cA$ is a sheaf of algebras over $X$, $\Om$ is an $\cA-$module and $\prt : \cA \ra \Om$ is a {\em Leibniz morphism}, i.e., a $\C-$linear sheaf morphism, that satisfies the {\em Leibniz condition\/}:
 \[
 \prt(\gra\beta) = \gra\prt(\beta) + \beta\prt(\gra), \qquad
 (\gra, \beta) \in \cA \times_X \cA,
  \]
  where ``$\times_X$" means fiber product over $X$.
\end{Def}

\begin{exas} {\bf (1) Smooth manifolds}. Every smooth manifold $X$ defines a differential triad
 \beq
 \grd^{\infty}_X = (\cC^{\infty}_X, d_X, \Om_X^1).
 \eeq
In this respect, $\cC^{\infty}_X$ is the structure sheaf of germs of smooth $\C$--valued functions on $X$; $\Om_X^1$ is the sheaf of germs of its smooth $\C$--valued 1--forms, namely, it consists of the smooth sections of the complexification of the cotangent bundle; and, $d_X$ is the sheafification of the usual differential. For details we refer to \cite[Vol. II, p. 9]{MAL;VS}.
We shall call (2.1) the {\em smooth differential triad} of $X$.

\smallskip
Apart of the smooth differential triads over manifolds, the above abstraction includes also includes differential triads on arbitrary topological spaces. We give a brief account of some basic examples. Details are found in \cite{MP;cdt}.

\smallskip
\noindent{\bf (2) Projective limits of manifolds}. It is known that the projective limit of a projective system of manifolds is not necessarily a manifold. M.E. Verona \cite{VERO} introduced a class of functions on such limits and defined a differential on these functions, in order to apply differential geometric considerations to the limits. Verona's construction is a (real) differential triad.

\smallskip
\noindent {\bf (3) Sheafification of K\"ahler's differential.}
Assume that $A$ is a unital, commutative and associative $\C$-algebra. We denote by $\mu$ the algebra multiplication, by $\phi$ the canonical bilinear map $\phi(x,y) = x \otimes y$ and by $m$ the linear map that corresponds to $\mu$, that is,
 \[
 (m \circ \phi)(x,y) = m(x \otimes y) = \mu(x,y) = xy,
 \]
for every $(x,y) \in A \times A$; we also set $I := \ker m$. Then $I$ is an ideal of $A \otimes_{\C} A$ and the map
 \[
 \grd_A : A \lra I/I^2 : x \lmt (x\otimes 1 - 1\otimes x) + I^2.
 \]
is a derivation (: {\em K\"ahler's differential}) (\cite[A.III. p.132]{BOURBA}).

Let now $(A_U, r^U_V)$ be a presheaf of algebras of the above type over an arbitrary topological space $X$, generating a sheaf $\cA$. Then the family $(A_U \otimes_{\C} A_U)$ generates $\cA \otimes_{\C} \cA$. The respective families of maps $(\mu_U)$, $(\phi_U)$ and $(m_U)$ are compatible with the presheaf restrictions, hence the family of kernels $(I_U := \ker m_U)$ is a presheaf of  $A(U)$--submodules of $(A_U \otimes_{\C} A_U)$ generating a sheaf ${\mathcal I}$. Besides,  $\grd_{\cA} = (\grd_{A_U})$ is proved to be a presheaf morphism whose factors are derivations. According to our definition, the triplet $(\cA, \grd_{\cA}, {\mathcal I}/{\mathcal I}^2)$ is a differential triad, called {\slb the sheafification of K\"ahler's differential}.

\smallskip
\noindent {\bf (4) Differential spaces.} Differential spaces and the subsequent differential-geometric concepts on them have been introduced by R. Sikorski (\cite{SIK1, SIK2}). Their sheaf-theoretic generalization, due to M. A. Mostow (\cite{MOST}), defines a differential triad.

\smallskip
\noindent {\bf (5) Differentiable spaces.} Here by ``differentiable spaces'' we refer to Spallek's $\infty$-standard differential spaces (see, for instance, \cite{Spal}), described below: An $\R$-algebra $A$ is a {\slb differentiable algebra}, if there is some $n \in \N$ and some closed ideal $\mathfrak{a}$ of $\8(\R^n)$, so that $A$ is (algebraically) isomorphic to the quotient
 \[
 A \cong \8(\R^n)/\mathfrak{a}.
 \]
Let $\frak{M}(A)$ be the spectrum of such an algebra. We denote by $A_U$ the ring of (equivalence classes of) fractions $\frac{a}{s}$, with $a, s \in A$ and $s(x) \neq 0$, for every $x \in U$. Then $(A_U)_{U \in \tau_{\frak{M}(A)}}$ is a presheaf of algebras on $\frak{M}(A)$ generating a sheaf $\tilde{A}$ called {\slb the structural sheaf on} $\frak{M}(A)$.

Now a pair $(X,{\mathcal O})$, where ${\mathcal O}$ is a sheaf of algebras over a topological space $X$, is called a {\slb differentiable space}, if every point $x \in X$ has an open neighbourhood $U$ such that $(U,{\mathcal O}_U)$ is isomorphic to a pair $(\frak{M}(A), \tilde{A})$, as above.

The sheafification of K\"ahler's differential provides the spectrum $\frak{M}(A)$ of any differentiable algebra $A$ with a differential triad $(\tilde{A}, \prt_A, \Om_A)$, and the local coincidence of a differentiable space $(X,{\mathcal O})$ with some $(\frak{M}(A),\tilde{A})$ provides an ${\mathcal O}$-module $\Om$ and a sheaf morphism $\prt : {\mathcal O} \ra \Om$, so that $({\mathcal O},\prt,\Om)$ is a differential triad over $X$; for details, see \cite{LNotes}.

\smallskip
\noindent {\bf (6) Differential algebras of generalized functions.}
A special case of an algebra which is a quotient of a functional algebra by a certain ideal has been defined by E. E. Rosinger \cite{ROS1}: Let $\emptyset \neq X \sst \R^n$ open. The corresponding {\slb nowhere dense differential algebra of generalized functions} on $X$ is the quotient
 \[
 A_{\rm{nd}}(X) := (\8(X,\R))^{\N}/{\mathcal I}_{\rm{nd}}(X)
 \]
where ${\mathcal I}_{\rm{nd}}(X)$ is the {\slb nowhere dense ideal} consisting of all the sequences $w = (w_m)_{m \in \N}$ of smooth functions $w_m \in \8(X,\R)$ which satisfy an ``{\slb asymptotic vanishing property}''.
Then $A_{\rm{nd}} = (A_{\rm{nd}}(U))_{U \in \tau_X}$ with the obvious restrictions is a presheaf of associative, commutative, unital algebras, inducing a (fine and flabby) sheaf $\cA$ on $X$. Next, an $\cA$-module $\Om$ is defined: for every $U \in \tau_X$, $\Om(U)$ is the free $A_{\rm{nd}}(U)$-module of rank $n$, with the free generators $d_i x_1, \dots, d_i x_n$. Consequently, the elements of $\Om(U)$ take the form
 \[
 \sum_{i=1}^n V_i \, d_i x_i,
 \]
where $V_i \in A_{\rm{nd}}(U)$. Finally, the differential $\prt : \cA \ra \Om$ is defined by the presheaf morphism $(\prt_U)_{U \in \tau_X}$, with
 \[
 \prt_U(V) = \sum_{i=1}^n \prt_i(V) \, d_i x_i
 \]
where $\prt_i$ denotes the usual $i$-partial derivation. Hence a differential triad $(\cA, \prt, \Om)$ is obtained {\em ``including the largest class of singularities dealt with so far''}. The algebras consisting the present structure sheaf {\em ``contain the Schwartz distributions''}, while they also {\em ``provide  global solutions for arbitrary analytic nonlinear PDEs. Moreover, unlike the distributions, and as a matter of physical interest, these algebras can deal with the vastly larger class of singularities which are concentrated on arbitrary closed, nowhere dense subsets and, hence, can have an arbitrary large positive Lebesgues measure''} (\cite[Abstract]{MAL-ROS1}).
\end{exas}

\medskip
Due to specificities of sheaf theory and the adjunction of the functors $f_*$
and $f^*$ induced by a continuous map $f$, there are three equivalent ways to
introduce the notion of a morphism of differential triads (cf. \cite{MP;Mor,
MP;pre-Lie, MP;calc}). The one in \cite{MP;Mor} is a
straightforward generalization of the situation we have in the theory of
manifolds, and it is the most suitable for our purposes here. First we notice
that if $\grd_X := \dtr{X}$ is a differential triad over $X$ and $f : X \ra Y$
is a continuous map, then {\em the push-out of\/ $\grd_X$ by $f$}
 \[
 f_*(\grd_X) \equiv (f_*(\cA_X), f_*(\prt_X), f_*(\Om_X))
 \]
is a differential triad over $Y$.

\begin{Def}
Let $\grd_X = \dtr{X}$ and $\grd_Y = \dtr{Y}$ be differential triads over the
topological spaces $X$ and $Y$, respectively. A {\em morphism of differential
triads} $\hat{f} : \grd_X \ra \grd_Y$ is a triplet $\hat{f} = \morf$, where

(i) $f : X \ra Y$ is continuous;

(ii) $f_{\cA} : \cA_Y \ra f_*(\cA_X)$ is a unit preserving morphism of sheaves
of algebras;

(iii) $f_{\Om} : \Om_Y \ra f_*(\Om_X)$ is an $f_{\cA}-$morphism, namely, it is a morphism of sheaves of additive groups, satisfying
 \[
 f_{\Om}(aw) = f_{\cA}(a)f_{\Om}(w), \  \2q \forall \ (a,w) \in \cA_Y \times_Y \Om_Y;
 \]

(iv) The diagram
 \begin{equation*}
 \begin{diagram}
 \cA_Y & \rTo^{f_{\cA}} & f_*(\cA_X) \\
 \dTo^{\prt_Y} & & \dTo_{f_*(\prt_X)} \\
 \Om_Y & \rTo_{f_{\Om}} & f_*(\Om_X)
 \end{diagram} \tag*{\text{\sc{Diagram 1}}}
 \end{equation*}
is commutative.

Extending the standard terminology, we shall say that a continuous map $f : X \ra Y$ is {\em differentiable}, if it is completed into a morphism $\hat{f} = \morf$.
Besides, we say that $f_{\Om}$ is a {\em differential} of $f$.
\end{Def}

\medskip
If $\grd_X, \grd_Y, \grd_Z$ are differential triads over the topological spaces $X,Y,Z$, respectively, and $\hat{f} = \morf : \grd_X \ra \grd_Y$, $\hat{g} =
\morg : \grd_Y \ra \grd_Z$ are morphisms, setting
 \begin{equation}
 \begin{gathered}
 (g \circ f)_{\cA} := g_*(f_{\cA}) \circ g_{\cA} \ \text{ and } \
 (g \circ f)_{\Om} := g_*(f_{\Om}) \circ g_{\Om}
 \end{gathered}
 \end{equation}
we obtain a morphism
 \begin{equation}
 \widehat{g \circ f} = (g \circ f, (g \circ f)_{\cA}, (g \circ f)_{\Om}) :
 \grd_X \ra \grd_Z.
 \end{equation}
The differential triads, their morphisms and the {\em composition law} defined
by (2.2) and (2.3) form a category, which will be denoted by $\DT$. Note that
the identity $\id_{\grd}$ of a differential triad $\grd = \dtr{}$ over $X$ is
the triplet $(\id_X, \id_{\cA}, \id_{\Om})$.

\begin{exa}
Consider the smooth manifolds $X$ and $Y$ provided with their smooth
differential triads $\grd^{\infty}_X = (\cC^{\infty}_X, d_X, \Om_X^1)$ and
$\grd^{\infty}_Y = (\cC^{\infty}_Y, d_Y, \Om_Y^1)$, respectively, and let $f :
X \ra Y$ be a smooth map. Then, for every $V \sst Y$ open, set
 \begin{gather*}
 \cC^{\infty}_Y(V) \equiv \cC^{\infty}(V,\C) \ \text{ and} \\
 f_*(\cC^{\infty}_X)(V) : = \cC^{\infty}_X(f^{-1}(V)) \equiv \cC^{\infty}(f^{-1}(V),\C).
 \end{gather*}
The map
 \beq
 f_{\cA V} : \cC^{\infty}(V) \lra \cC^{\infty}(f^{-1}(V)) : \gra \lmt \gra \circ f
 \eeq
is a unit preserving algebra morphism, while the family $(f_{\cA V})_V$ is a
presheaf morphism giving rise to a unit preserving algebra sheaf morphism
$f_{\cA} : \cC^{\infty}_Y \ra \cC^{\infty}_X$. On the other hand, the respective tangent map $Tf : TX \ra TY$ defines the so--called {\em pull--back of the smooth 1--forms by $f$}
 \beq
 \begin{gathered}
 f_{\Om V} : \Om_Y^1(V) \lra \Om_X^1(f^{-1}(V)) :
 \om \lmt \om \circ Tf.
 \end{gathered}
 \eeq
Note that $\om \circ Tf$ is a standard notation in CDG, with
 \[
 (\om \circ Tf)_x(u) = \om_{f(x)}(T_xf(u)), \qquad x \in X, \ u \in T^{\C}_xX,
 \]
where $T^{\C}_xX$ is the complexification of the tangent space of $X$ at $x$ and $T_xf$ stands also for the extension of the tangent map $T_xf$ on $T^{\C}_xX$. Then $f_{\Om V}$ is an $f_{\cA V}-$morphism and the family $(f_{\Om V})_V$ is a presheaf morphism yielding an $f_{\cA}-$morphism $f_{\Om} : \Om^1_Y \ra f_*(\Om_X^1)$. Note that, if $(V,\psi)$ is a chart of $Y$ with coordinates $(y_1,\dots,y_n)$, and $\om \in \Om_Y^1(V)$, then there are $\gra_i \in \8(V,\C)$, $i = 1,\dots,n$, with $\om = \sum_{i=1}^n \gra_i \cdot d_Yy_i$. In this case the pull--back of $\om$ by $f$ is given by
 \beq
 f_{\Om V}(\om) = \sum_{i=1}^n (\gra_i \circ f) \cdot (d_Yy_i \circ Tf).
 \eeq
The commutativity of Diagram 1 is equivalent to the chain rule, therefore
 \[
 \morf : \grd^{\infty}_X \ra \grd_Y^{\infty}
 \]
is a morphism in $\DT$.
\end{exa}

Clearly, if $\cM an$ is the category of smooth manifolds, the functor
 \beq
 F : \cM an \ra \DT,
 \eeq
where $F(X)$ is the smooth differential triad $\grd_X^{\infty}$ and $F(f) = \morf$, described in Example 2.4, is an embedding.

\section{Existence and uniqueness of morphisms}

In the above abstract approach of differentiability two problems arise:

(1) For arbitrary algebra sheaves, the existence of a morphism extending a map is not ensured, even for very simple mappings (for instance, the constant map);

(2) likewise, the uniqueness of a morphism over a fixed $f$ (the
analogue of the uniqueness of differentials) is not ensured.

\medskip
Regarding (1), let $X, Y$ be topological spaces provided by the differential triads $\grd_X$ and $\grd_Y$ and let $c$ denote some fixed element in $Y$ and $c : X \ra Y$ the respective constant map. Then, for every $V \sst Y$ open, with $c \in V$,
 \[
 c_*(\cA_X)(V) = \cA_X(c^{-1}(V)) = \cA_X(X),
 \]
while, if $c \notin V$, then $c_*(\cA_X)(V) = \emptyset$. Thus $c_*(\cA_X)$ is
a sheaf over the space $\{c\}$ (at least for $T_{1}-$spaces), whose (unique) stalk is the space of global sections $\cA_X(X)$, and the question of differentiability of the constant map $c$ reduces to the question of the existence of a unit preserving algebra morphism
 \[
 c_{\cA} : \cA_{Y,c} \lra \cA_X(X),
 \]
and, similarly, of a $c_{\cA}-$morphism
 \[
 c_{\Om} : \Om_{Y,c} \lra \Om_X(X),
 \]
making Diagram 1 commutative. Note that $\cA_{Y,c}, \Om_{Y,c}$ stand for the corresponding stalks of $\cA_{Y}, \Om_{Y}$ at $c$. The existence of a non--trivial (: unit preserving) algebra morphism is not assured, of course,  in the general case.

However, we obtain the differentiability of the constant map, if the sheaf
$\cA_Y$ is functional. In this respect, we recall that a {\em functional
algebra sheaf} over $Y$ (see \cite[Vol.~I, p.~49]{MAL;Abstract} and \cite{MAL;CM}) we mean a sheaf of algebras which is a subsheaf of the sheaf
$\cC_Y$ of germs of continuous $\C-$valued functions on $Y$. In this respect, we have

\begin{pro}
If $\grd_I = \dtr{I}$ are differential triads over the spaces $I = X, Y$ and
$\cA_Y$ is functional, then every constant map $c : X \ra Y$ is differentiable.
\end{pro}

\begin{proof}
Since $\cA_X$ has a unit, it contains the constant sheaf $X \times \C$. Thus
every $k \in \C$ can be considered as a global section of $\cA_X$. Besides, for every $a \in \cA_{Y,c}$, there is an open neighborhood $V$ of $c$ in $Y$ and
$\gra \in \cA_Y(V) \sst \cC(V,\C)$, with $a = [\gra]_c$. Setting
 \[
 c_{\cA} : \cA_{Y,c} \lra \cA_X(X) : a \lmt \gra(c)
 \]
we obtain a unit preserving algebra morphism.

On the other hand, $c_{\cA}(a) \in \C$ implies $\prt_X \circ c_{\cA} = 0$. Thus the zero morphism $0 : \Om_{Y,c} \ra \Om_X(X)$ completes the triplet $(c,
c_{\cA}, c_{\Om} = 0)$, so that Diagram 1 is commutative.
\end{proof}

\smallskip
Concerning (2), although there may be infinitely many pairs $(f_{\cA},
f_{\Om})$ making $f$ differentiable, due to the commutativity of Diagram 1,
in certain cases the pairs need to satisfy some restrictions: For instance, if
the sheafification of K\"ahler's differential is considered, then $f_{\cA}$ determines $f_{\Om}$ (see \cite[Vol.~II, p.~327]{MAL;VS}
and \cite{MP;cdt}). More generally, we have the following

\begin{pro}
If $\morf$ is a morphism in $\DT$, then $f_{\Om}$ is uniquely determined by
$f_{\cA}$ on the image ${\rm Im} \,\prt_Y$ of $\prt_Y$.

Conversely, if $\prt_X$ vanishes only on the constant subsheaf $X \times \C
\sst \cA_X$, then $f_{\cA}$ is uniquely determined by $f_{\Om}$.
\end{pro}

\begin{proof}
The former assertion is obvious. Regarding the latter, if $\morf$ and $(f,
f'_{\cA},f_{\Om})$ are morphisms in $\DT$, the equality
 \[
 f_{\Om}(\prt_Y(a)) = f_*(\prt_X)(f_{\cA}(a)) = f_*(\prt_X)(f'_{\cA}(a))
 \]
implies
 \[
 f_*(\prt_X)(f_{\cA}(a)-f'_{\cA}(a)) = 0.
 \]
Consequently, $f_{\cA}(a)-f'_{\cA}(a) = c \in \C$, for every $a \in \cA_Y$.
Since $f_{\cA}-f'_{\cA}$ is $\C-$linear, $f_{\cA}-f'_{\cA} = 0$.
\end{proof}

\section{Morphisms over manifolds}

In this section we need some concepts from the general theory of (non--normed) topological algebras (see \cite{MAL;TA}). We briefly introduce them. A \emph{topological algebra} is a complex associative algebra $\A$, which is also a topological vector space such that the ring multiplication in $\A$ is separately continuous. A topological algebra $\A$ with a unit element is called a \emph{$Q-$algebra} if the group $G_{\A}$ of its invertible elements is open (ibid., p.~139). The algebra $\cC ^\infty [0, 1]$ of all smooth functions on $[0, 1]$ is a $Q-$algebra \cite[Example 6.23(3)]{fra}. A topological algebra $\A$ whose topology is defined by a directed family of submultiplicative seminorms is called \emph{locally $m-$convex}. The preceding algebra is of this kind. If $\A$ is an algebra, a non--zero complex multiplicative linear functional of $\A$ is called \emph{character of $\A$}. If $\A$ is a commutative locally $m-$convex algebra with unit, we denote by $\frak M (\A)$ the \emph{topological spectrum} or simply \emph{spectrum} of $\A$, consisting of all continuous characters of $\A$; the spectrum endowed with the relative weak$^*$ topology from the topological dual of $\A$, turns into a Hausdorff completely regular topological space. \emph{The spectrum of a commutative locally $m-$convex $Q-$algebra with unit is always a compact space} (ibid., p.~87, Lemma 1.3). Moreover, \emph{each character of a $Q-$algebra is continuous} \cite[p.~72, Corol. 7.3]{MAL;TA}.

As we have seen in Example 2.2, every smooth manifold $X$ gives rise to a
differential triad $\grd^{\infty}_X = (\8_X, d_X, \Om_X^1)$, while a smooth map $f : X \ra Y$ between such manifolds is completed to the morphism $F(f)$ given
in Example 2.4. But there rises the following question: If $f : X \ra Y$ is a
{\em continuous} map between smooth manifolds, {\em can it be differentiable in the abstract setting, without being such in the classical sense?}

The answer is no. This is due, on the one hand, to the fact that every character of a $Q$--algebra is automatically continuous and on the other hand, to the fact that the continuous characters of our algebras of smooth functions are uniquely determined by point evaluations, corresponding to the elements of the domain of the respective  smooth functions (see \cite[p.~56, (4.43)]{fra}, \cite[p.~227, (2.6)]{MAL;TA}. This implies that all unit preserving morphisms, whose domain is a suitable function $Q-$algebra have the form (2.4), for an appropriate $f$ between the spectra of the topological algebras involved. In our case, for a smooth manifold $Y$, the algebra $\8(Y)$ is a \emph{Fr\'echet} (i.e., metrizable and complete) locally $m-$convex algebra \cite[p.~131, discussion around (4.19)]{MAL;TA}. In particular, its spectrum is homeomorphic to $Y$ (see \cite[Example 4.20(2)]{fra}). Thus, if $Y$ is compact, $\8(Y)$ is a $Q$--algebra \cite[p.~143, Prop. 1.1; p. 183, Corol. 1.1 and p.~187, Lemma 1.3]{MAL;TA} and every algebra morphism $h : \8(Y) \ra \8(X)$ takes the form (2.4), for a suitable $f$ between the spectra of the above algebras, i.e., between $X$ and $Y$. In the case when $Y$ is not compact, the nice way that sheaf morphisms localize does the trick (see Theorem 4.4, below).

\medskip
Let $\cS$ be a sheaf over $X$. For every open $U \sst X$, $\cS(U)$ is the set of all sections of $\cS$ over $U$. If $K \sst X$ is closed, we denote by  $\cS(K)$ the inductive limit of all $\cS(U)$, with $U \sst X$ open and $K \sst U$, i.e.,
 \beq
 \cS(K) := \varinjlim_{K \sst U} \cS(U).
 \eeq
Besides, if $f : \cS \ra \cT$ is a sheaf morphism, we denote by
 \beq
 f_K := \varinjlim_{K \sst U} f_U : \cS(K) \lra \cT(K)
 \eeq
the inductive limit of all $f_U : \cS(U) \ra \cT(U)$, with $U$ as above. If $A, B$ are open sets in
$U$ with $B \subseteq A$, by definition $(r^A_B)$ and $(\rho^A_B)$ are the restrictions of the presheaves of sections of $\cS$ and $\cT$, respectively, and
 \begin{gather*}
 r^V_K : \cS(V) \lra \cS(K), \quad \rho^V_K : \cT(V) \lra \cT(K)
 \end{gather*}
are the canonical maps. Then the following diagram
 \begin{equation*}
 \begin{diagram}
 \cS(V) & \rTo^{f_V} & \cT(V) \\
 \dTo^{r^V_K} & & \dTo_{\rho^V_K}\\
 \cS(K) & \rTo_{f_K} & \cT(K)
 \end{diagram} \tag*{\text{\sc{Diagram 2}}}
 \end{equation*}
commutes. That is,
 \beq
 f_K([\gra]_K) = f_K \circ r^V_K(\gra) = \rho_K^V \circ f_V(\gra) = [f_V(\gra)]_K,
 \eeq
for every $V \subseteq X$ open, with $K \sst V \sst U$, and every $\gra \in \cS(V)$.

Suppose now that $K$ as before is compact. Consider the inductive system $\{\8_X(V)\}_{V}$, $V \subseteq X$ open, with $K \sst V$ (where the connecting maps are the obvious ones) and put $$\8_X(K):= \varinjlim_{K \sst V} \8_X(V).$$ In this regard, we have

\begin{pro}
Let $X$ be a smooth manifold and let $\8_X$ denote the sheaf of germs of smooth $\C$--valued functions on $X$. Let $(U,\phi)$ be a chart on $X$ and $K \sst U$
compact. Then $\8_X(K)$ is a locally $m-$convex $Q-$algebra, whose spectrum
coincides with $K$.
\end{pro}

\begin{proof}
Since the families $\{V \sst X \ \text{open} : K \sst V\}$ and $\{V \sst U \ \text{open} : K \sst V \}$ are cofinal, we have
 \[
 \8_X(K) := \varinjlim_{K \sst V} \8_X(V) = \varinjlim_{K \sst V \sst U} \8_X(V).\
 \]
Every open $V$ with $K \sst V \sst U$ is still the domain of a chart, namely of $(V,\phi|_V)$, thus on every $\8_X(V)$ one can define the family of seminorms
 \beq
 N_{m,C}(f) := \sup_{|p|\leq m}\big(\sup_{x \in C}|D^p(f \circ
 \phi^{-1})(\phi(x))|\big),
 \eeq
for every $f \in \8_X(V)$, where $C$ is a compact subset of $V$, $m \in \N \cup \{0\}$,
$|p|$ stands for the length $|p| = p_1+\cdots+p_n$ of the multi--index $p =
(p_1,\dots,p_n)$ and
 \[
 D^p(f\circ \phi^{-1})(\phi(x)) =
 \left.\frac{\prt^{|p|}}{\prt u_1^{p_1} \cdots \prt u_n^{p_n}}(f\circ
 \phi^{-1})\right|_{\phi(x)}
 \]
(see \cite[p.~130, (4.14)]{MAL;TA}). Topologized by the above family of seminorms,
$\8_X(V)$ becomes a \emph{Fr\'echet locally $m-$convex algebra} (ibid., p.~130, (4.12)). We consider the inductive limit algebra $\8_X(K)$ topologized by the {\em locally $m-$convex inductive limit topology}, that is, the finest locally $m-$convex topology on $\8_X(K)$, making all the canonical maps
 \[
 r_K^V : \8_X(V) \ra \8_X(K)
 \]
continuous (ibid., p.~120, Def. 3.1); thus $\8_X(K)$ becomes a locally $m-$convex algebra.

We set
 \[
 N_K : \8_X(K) \lra \C : a \lmt N_{0,K}(\gra),
 \]
where $a = [\gra]_K$, with $\gra \in \8_X(V)$, for some open $V$ in $X$, with $K \sst V \sst U$ and $N_{0,K}$ is given by (4.4). Then $N_K$ is a submultiplicative seminorm on $\8_X(K)$, with
 \[
 N_K \circ r^V_K = N_{0,K} : \8_X(V) \lra \R,
 \]
for every $V$ as above, hence the topology induced by $N_K$ on $\8_X(K)$ is a locally $m-$convex topology making all canonical maps $r^V_K$ continuous. Consequently, it is coarser than the locally $m-$convex inductive limit
topology.

We prove that the topology defined by $N_K$ makes $\8_X(K)$ a $Q-$algebra: Let $a \in \8_X(K)$ be invertible and let $b$ be its inverse. Then there are an open set $V$ in $X$, with $K \sst V \sst U$, and $\gra$, $\beta \in \8_X(V)$, such that $a = [\gra]_K$, $b = [\beta]_K$ and $\gra \cdot \beta = 1|_V$. Since $\gra$ is invertible on $V \supseteq K$, $|\gra|$ takes a least value
 \[
 \gre := \min\{|\gra(y)| : y \in K\}> 0.
 \]
Then the open ball $B(a,\gre/2)$ with respect to $N_K$ is an open neighborhood of $a$ contained in the group $G_{\8_X(K)}$ of the invertible elements of $\8_X(K)$. This makes $G_{\8_X(K)}$ open with respect to the $N_{K}-$topology and, consequently, to the locally $m-$convex inductive limit topology.

Regarding the spectrum of $\8_X(K)$, we have (see \cite[p.~156, Theor. 3.1]{MAL;TA} and \cite[Example 4.20(2)]{fra}, for the spectrum of $\8_X(V)$)
 \[
 \frak{M}(\8_X(K)) = \varprojlim_{K \sst V \sst U} \frak{M}(\8_X(V)) = \varprojlim_{K \sst V \sst U} V = \bigcap_{K \sst V \sst U} V = K,
 \]
which completes the proof.
\end{proof}

Taking into account that every character of a $Q-$algebra is continuous (see discussion at the beginning of this section), we obtain

\begin{cor}
Let $X$ be a smooth manifold and let $\8_X$ denote the sheaf of germs of smooth $\C-$valued functions on $X$. Then, for every $x \in X$, the stalk $\8_{X,x}$ topologized with the locally $m-$convex inductive limit topology is a locally $m-$convex $Q-$algebra, with
 \[
 \mathfrak{M}(\8_{X,x}) = \{ \, x \, \}.
 \]
As a result, $\8_{X,x}$ has a unique character, which is also continuous.
\end{cor}

If $x \in V$, the usual evaluation map at $x$,
 \[
 ev^V_x : \8_X(V) \lra \C : \gra \lmt \gra(x)
 \]
is a continuous character of $\8_X(V)$ \cite[p.~58]{fra}, and the family $(ev^V_x)_V$ commutes with the presheaf restrictions, hence the inductive limit map
 \beq
 ev_x := \varinjlim_{x \in V} ev^V_x : \8_{X,x} \lra \C
 \eeq
exists and it is a (continuous) character of the stalk $\8_{X,x}$. According to the preceding corollary, it is the unique (continuous) character in $\mathfrak{M}(\8_{X,x})$. That is, we have the following

\begin{cor}
The unique character of $\8_{X,x}$ coincides with that given by \emph{(4.5)}.
\end{cor}

\begin{thm}
Let $X$, $Y$ be smooth manifolds and let $f : X \ra Y$ be a continuous map. If there is a unit preserving morphism of algebra sheaves $f_{\cA} : \8_Y \ra f_*(\8_X)$, then $f$ is smooth and
 \beq
 f_{\cA V}(\gra) = \gra \circ f,
 \eeq
for every $V \sst Y$ open and every $\gra \in \8_Y(V)$.
\end{thm}

\begin{proof}
Let $x \in X$ and $(U,\phi)$ a chart of $Y$ at $f(x)$. Consider the
stalks
 \begin{gather*}
 \8_{Y, f(x)} = \varinjlim_{f(x) \in V \sst U} \8_Y(V)\,, \\
 f_*(\8_X)_{f(x)} = \varinjlim_{f(x) \in V \sst U} \8_X(f^{-1}(V))
 \end{gather*}
and the algebra morphism
 \[
 f_{\cA, f(x)} := \varinjlim_{f(x) \in V \sst U} f_{\cA V}
 \ : \ \8_{Y,f(x)} \lra f_*(\8_X)_{f(x)}
 \]
(cf. (4.1) and (4.2)), where $(f_{\cA V} : \8_Y(V) \ra f_*(\8_X)(V))_V$ is the presheaf morphism induced by $f_{\cA}$ (see (2.4)). We now set
 \[
 \overline{ev}_x := \varinjlim_{f(x) \in V \sst U} ev_x^{f^{-1}(V)} : f_*(\8_X)_{f(x)} \lra \C,
 \]
where $ev^{f^{-1}(V)}_x : \8_X(f^{-1}(V)) \ra \C$ is the usual evaluation at $x \in f^{-1}(V)$. Then $\overline{ev}_x$ is an algebra morphism, therefore,
 \[
 \overline{ev}_x \circ f_{\cA,f(x)} : \8_{Y,f(x)} \lra \C
 \]
is a character of $\8_{Y,f(x)}$. By virtue of Corollary 4.3, this character coincides with the one given by
 \[
 ev_{f(x)} := \varinjlim_{f(x) \in V \sst U} ev^V_{f(x)} : \8_{Y,f(x)} \lra \C.
 \]
Further, if $(r^A_B)$ and $(\rho^A_B)$, $A, B$ open in $Y$ with $B \sst A$, are the restrictions of the presheaves of sections of $\8_Y$ and of $f_*(\8_X)$, respectively, by the definition of $ev_{f(x)}$, we have
 \[
 ev_{f(x)}([\gra]_{f(x)}) = ev_{f(x)}\circ r^V_{f(x)}(\gra) = ev^V_{f(x)}(\gra) = \gra(f(x))
 \]
for every $V \subseteq Y$ open with  $f(x) \in V$ and every $\gra \in \8_Y(V)$, while
 \begin{align*}
 (\overline{ev}_x \circ f_{\cA,f(x)})([\gra]_{f(x)})
 & = (\overline{ev}_x \circ f_{\cA,f(x)})(r^V_{f(x)}(\gra)) \\
 & = \overline{ev}_x \circ \rho^V_{f(x)} \circ f_{\cA V}(\gra) = ev^{f^{-1}(V)}_x(f_{\cA V}(\gra)) \\
 & = (f_{\cA V}(\gra))(x).
 \end{align*}
That is,
 \[
 (f_{\cA V}(\gra))(x) = \gra(f(x)), \ \ \ \forall \ x \in X,
 \]
which proves (4.6).

Regarding the smoothness of $f$, it suffices to notice that, for every $V \sst Y$ open and every $\gra \in\8_Y(V)$, we have that
 \[
 \gra \circ f = f_{\cA V}(\gra) \in \8_X(f^{-1}(V)),
 \]
so $\gra \circ f $ is smooth, and this completes the proof.
\end{proof}

\begin{thm}
Let $X$, $Y$ be smooth manifolds provided with their smooth differential triads $\grd_X^{\infty}$ and $\grd_Y^{\infty}$. Besides, let $\hat{f } = \morf :
\grd_X^{\infty} \ra \grd_Y^{\infty}$ be a morphism in $\DT$. Then $f$ is smooth in the ordinary sense and $\hat{f} = F(f)$, where $F$ is the embedding \emph{(2.7)}.
\end{thm}

\begin{proof}
By the preceding Theorem 4.4, $f$ is smooth and the presheaf morphisms
$(f_{\cA V})_V$ are given by (2.4).

In order to prove that $f_{\Om V}$ satisfies (2.6), it suffices to prove it for the domains of the charts of the maximal atlas of $Y$, since they form a basis
for its topology. Thus let $(V,\psi)$ be a chart with coordinates
$(y_1,\dots,y_n)$, and let $\om = \sum_{i=1} ^n \gra_i \cdot dy_i \in \Om_Y^1(V)$. Applying
the commutativity of the diagram
 \begin{equation*}
 \begin{diagram}
 \8(V,\C) & \rTo^{f_{\cA V}} & \8(f^{-1}(V),\C) \\
 \dTo^{d_{YV}} & & \dTo_{d_{Xf^{-1}(V)}} \\
 \Om_Y^1(V) & \rTo_{f_{\Om V}} & \Om_X^1(f^{-1}(V)
 \end{diagram} \tag*{\text{\sc{Diagram 3}}}
 \end{equation*}
we obtain
 \begin{align*}
 f_{\Om V}(\om) & = f_{\Om V}(\sum_{i=1} ^n \gra_i \cdot  dy_i) =
           \sum_{i=1} ^n f_{\cA V}(\gra_i) \cdot f_{\Om V}(dy_i) \\
                & = \sum_{i=1} ^n (\gra_i \circ f) \cdot (f_{\Om V} \circ
                         d_{YV})(y_i) \\
                & = \sum_{i=1} ^n (\gra_i \circ f) \cdot (d_{Xf^{-1}(V)} \circ
                         f_{\cA V})(y_i) \\
                & = \sum_{i=1} ^n (\gra_i \circ f) \cdot d_{Xf^{-1}(V)}(y_i \circ f) \\
                & = \sum_{i=1} ^n (\gra_i \circ f) \cdot (d_{YV}y_i \circ Tf)
\end{align*}
which proves our assertion.
\end{proof}

The last two results imply the following

\begin{thm}
$\cM an$ is a full subcategory of $\DT$. In other words, when smooth mani\-folds $X$ and $Y$ are considered, the sets of morphisms between them in the categories $\cM an$ and $\DT$ coincide; that is,
 \beq
 Hom_{\cM an}(X,Y) \cong Hom_{\DT}(X,Y).
 \end{equation}
\end{thm}

For the term ``full subcategory", see \cite[p.~15]{Mac}.
\medskip

{\em Final remark.} As we noticed in the introduction, ADG applies in large
scale phenomena (i.e. general relativity) when singularities appear \cite{MAL;ADG/GR+SING, MAL;JTP, MAL-ROS2} as well as in quantum mechanics \cite{MAL-ZAF}.
It also embodies phenomena usually studied by the ordinary CDG. Theorem 4.5
implies that ADG applied on the latter gives the same results with CDG.

\medskip
{\em Acknowledgement}. The authors wish to thank Professor A.~Mallios, for helpful comments on this work.

\begin{flushleft}
Department of Mathematics, University of Athens, Panepistimiopolis,\\ Athens 157~84, Greece\\
E--mail addresses: fragoulop@math.uoa.gr, \ mpapatr@math.uoa.gr
\end{flushleft}

\end{document}